\documentclass[11pt,a4paper,reqno]{amsart}
\usepackage[english]{babel}
\usepackage[applemac]{inputenc}
\usepackage[T1]{fontenc}
\usepackage{palatino}
\usepackage{amsmath}
\usepackage{amssymb}
\usepackage{amsthm}
\usepackage{amsfonts}
\usepackage{graphicx}
\usepackage{mathtools}

\DeclarePairedDelimiter\floor{\lfloor}{\rfloor}
\usepackage[colorlinks = true, citecolor = black]{hyperref}
\author{Tuomas Orponen}
\title{Projections of planar sets in well-separated directions}
\address{University of Helsinki, Department of Mathematics and Statistics}
\thanks{T.O. is supported by the Academy of Finland through the grant Restricted families of projections and connections to Kakeya type problems. The research was also partially supported by the European Research Council through Michael Hochman's ERC grant 306494.}
\email{tuomas.orponen@helsinki.fi}

\newcommand{\R}{\mathbb{R}}
\newcommand{\N}{\mathbb{N}}

\newcommand{\Z}{\mathbb{Z}}

\newcommand{\calT}{\mathcal{T}}

\newcommand{\calD}{\mathcal{D}}
\newcommand{\calH}{\mathcal{H}}

\newcommand{\calF}{\mathcal{F}}
\newcommand{\calS}{\mathcal{S}}

\newcommand{\spt}{\operatorname{spt}}

\newcommand{\calP}{\mathcal{P}}

\newcommand{\E}{\mathbb{E}}

\newcommand{\calE}{\mathcal{E}}

\newcommand{\spa}{\operatorname{span}}
\newcommand{\diam}{\operatorname{diam}}

\newcommand{\dist}{\operatorname{dist}}

\numberwithin{equation}{section}

\theoremstyle{plain}
\newtheorem{thm}[equation]{Theorem}
\newtheorem{conjecture}[equation]{Conjecture}
\newtheorem{lemma}[equation]{Lemma}

\newtheorem{cor}[equation]{Corollary}
\newtheorem{proposition}[equation]{Proposition}

\theoremstyle{definition}

\newtheorem{definition}[equation]{Definition}

\theoremstyle{remark}
\newtheorem{remark}[equation]{Remark}

\begin{document}

\begin{abstract} This paper contains two new projection theorems in the plane. 

First, let $K \subset B(0,1) \subset \R^{2}$ be a set with $\calH_{\infty}^{1}(K) \sim 1$, and write $\pi_{e}(K)$ for the orthogonal projection of $K$ into the line spanned by $e \in S^{1}$. For $1/2 \leq s < 1$, write
\begin{displaymath} E_{s} := \{e : N(\pi_{e}(K),\delta) \leq \delta^{-s}\}, \end{displaymath}
where $N(A,r)$ is the $r$-covering number of the set $A$. It is well-known -- and essentially due to R. Kaufman -- that $N(E_{s},\delta) \lessapprox \delta^{-s}$. Using the polynomial method, I prove that
\begin{displaymath} N(E_{s},r) \lessapprox \min\left\{\delta^{-s}\left(\frac{\delta}{r}\right)^{1/2},r^{-1}\right\}, \quad \delta \leq r \leq 1. \end{displaymath}
I construct examples showing that the exponents in the bound are sharp for $\delta \leq r \leq \delta^{s}$.

The second theorem concerns projections of $1$-Ahlfors-David regular sets. Let $A \geq 1$ and $1/2 \leq s < 1$ be given. I prove that, for $p = p(A,s) \in \N$ large enough, the finite set of unit vectors $S_{p} := \{e^{2\pi i k/p} : 0 \leq k < p\}$ has the following property. If $K \subset B(0,1)$ is non-empty and $1$-Ahlfors-David regular with regularity constant at most $A$, then
\begin{displaymath} \frac{1}{p} \sum_{e \in S_{p}} N(\pi_{e}(K),\delta) \geq \delta^{-s} \end{displaymath}
for all small enough $\delta > 0$. In particular, $\overline{\dim}_{\textup{B}} \pi_{e}(K) \geq s$ for some $e \in S_{p}$.
\end{abstract}

\maketitle

\section{Introduction}

Let $K \subset B(0,1) \subset \R^{2}$ be a compact set with $\calH_{\infty}^{1}(K) \sim 1$. For $0 \leq s \leq 1$, a classical result of R. Kaufman \cite{Ka}, sharpening the projection theorem of Marstrand \cite{Ma}, states that 
\begin{equation}\label{kaufman} \dim \{e \in S^{1} : \dim \pi_{e}(K) \leq s\} \leq s, \end{equation}
where $\pi_{e}$ denotes orthogonal projection onto $\spa(e)$ and $\dim$ is Hausdorff dimension. It seems unlikely that this bound is sharp for $s < 1$. It is conjectured in D. Oberlin's paper \cite{Ob} that the correct bound is $2s - 1$ instead of $s$, and \cite[Theorem 1.2]{Ob} corroborates this by showing that $\dim \{e : \dim \pi_{e}(K) < 1/2\} = 0$. A stronger, and significantly harder to prove, improvement to \eqref{kaufman} is due to Bourgain \cite{Bo}: a (non-trivial) application of his "discretised sum-product theorem" shows that the left hand side of \eqref{kaufman} tends to zero as $s \searrow 1/2$. However, even Bourgain's method of proof only gives an improvement to \eqref{kaufman} when $s$ is "very close" to $1/2$. So, for example, nothing better than \eqref{kaufman} is currently known for $s = 3/4$. 

The starting point of this paper was to investigate the case where $s$ is far away from $1/2$. In trying to prove statements about Hausdorff dimension, such as \eqref{kaufman}, a natural intermediary step is to find and solve a "$\delta$-discretised" analogue of the problem. In the current situation, the simplest such analogue is probably the following: fix $\delta > 0$, and let $E_{s}$ be the collection of vectors in $S^{1}$ such that $\pi_{e}(K)$ can be covered by $\leq \delta^{-s}$ intervals of length $\delta$. In symbols,
\begin{displaymath} E_{s} := \{e \in S^{1} : N(\pi_{e}(K),\delta) \leq \delta^{-s} \}, \end{displaymath}
where $N(A,r)$ the least number of $r$-balls required to cover $A$. How many $\delta$-intervals does it take to cover $E_{s}$? An argument close to Kaufman's proof of \eqref{kaufman} shows that 
\begin{equation}\label{discreteKaufman} N(E_{s},\delta) \lessapprox \delta^{-s}, \end{equation}
where $A \lessapprox B$ stands for $A \leq C\log(1/\delta) B$ for some absolute constant $C$. A significant difference between \eqref{kaufman} and \eqref{discreteKaufman} is, however, that $s$ is the best exponent in \eqref{discreteKaufman}, and the example proving this is extremely simple: one needs only take $K$ to be a horizontal unit line segment, and consider its projections (at scale $\delta$) on nearly vertical lines. It is worth emphasising that the bound \eqref{discreteKaufman} is even sharp for $s = 1/2$, whereas \eqref{kaufman} is not, according to Bourgain's result. 

So, the sharpness of \eqref{discreteKaufman} does not imply that \eqref{kaufman} is sharp; neither does it mean that the "$\delta$-discretised approach" to improving \eqref{kaufman} is doomed. However, one certainly needs to ask more subtle questions than "What is the best bound for $N(E_{s},\delta)$?". To find such questions, one can to consider the extremal configurations for \eqref{discreteKaufman}. It was already mentioned that the line segment exhibits worst-case behaviour, but this is surely not the only example: in fact, any union of $\sim \delta^{-s}$ parallel line segments of length $\delta^{s}$ works, as long as the union has large $1$-dimensional Hausforff content. 

Even if the examples $K$ extremal for \eqref{discreteKaufman} may be too diverse to classify, all the configurations I know of seem to have one feature in common: the associated $\approx \delta^{-s}$ directions in $E_{s}$ are very clustered. In the case of the horizontal line segment, for instance, they all lie packed around the vertical direction. Encouraged by this observation, a reasonable conjecture could be the following: if $E$ is any collection of $\sim \delta^{-s}$ vectors, which are "quantitatively not packed together", then $E$ contains a vector $e$ with $N(\pi_{e}(K),\delta) \geq \delta^{-s - \epsilon}$. Here is a more precise formulation:
\begin{conjecture}\label{mainConjecture} Assume that $K \subset B(0,1) \subset \R^{2}$ is a set with $\calH^{1}_{\infty}(K) \sim 1$. Let $E \subset S^{1}$ be any $\delta$-separated set of directions with cardinality $|E| \sim \delta^{-s}$, satisfying the non-concentration hypothesis
\begin{equation}\label{nonConc} |E \cap B(x,t)| \lesssim t^{\kappa}|E|, \qquad x \in S^{1}, \: t \geq \delta, \end{equation}
for some $\kappa > 0$. Then $N(\pi_{e}(K),\delta) \geq \delta^{-s - \epsilon}$ for some $e \in E$, where $\epsilon > 0$ is a constant depending only on $\kappa,s$. 
\end{conjecture}
The conjecture is true, and due to Bourgain, if $s$ is sufficiently close to $1/2$; in this case, one can also drop the \emph{a priori} assumption $|E| \sim \delta^{-s}$, because \eqref{nonConc} alone guarantees that $E$ contains enough directions, see \cite[Theorem 3]{Bo}. Progress in Conjecture \ref{mainConjecture} for a certain $s \in (1/2,1)$ would, most likely, lead to an improvement for the Hausdorff dimension estimate \eqref{kaufman} for the same $s$.

The first main result of the present paper is a variant of the conjecture, where the non-concentration hypothesis \eqref{nonConc} is replaced by the requirement that the vectors in $E$ be $r$-separated for some $\delta \leq r \leq 1$:

\begin{thm}\label{main} Let $K \subset B(0,1) \subset \R^{2}$ be a compact set with $\calH_{\infty}^{1}(K) \gtrsim 1$, and let $1/2 \leq s < 1$ and $\delta \leq r \leq 1$. Then
\begin{displaymath} N(E_{s},r) \lessapprox \min\left\{\delta^{-s}\left(\frac{\delta}{r}\right)^{1/2},\frac{1}{r}\right\}. \end{displaymath}
The exponents in the bound are sharp for $\delta \leq r \leq \delta^{s}$.
\end{thm}

\begin{remark} An equivalent formulation of Theorem \ref{main} -- more reminiscent of Conjecture \ref{mainConjecture} -- is the following: if $|E| \sim \delta^{-s}$, and the separation between the vectors in $E$ is at least $r \geq \delta$, then $N(\pi_{e}(K),\delta) \gtrapprox \delta^{-s}(r/\delta)^{1/2}$ for some $e \in E$. Assuming that $\delta \leq r < \delta^{s + \epsilon}$, a set $E$ satisfying these hypotheses can be found inside an arc of length $\sim \delta^{\epsilon}$, and such an $E$ naturally cannot satisfy the non-concentration hypothesis \eqref{nonConc} with $t = \delta^{\epsilon}$. So, in fact, the separation assumption in Theorem \ref{main} is neither weaker nor stronger than \eqref{nonConc}, and in particular Theorem \ref{main} gives new information even in the "$s$ is close to $1/2$" regime, which does not follow from Bourgain's paper \cite{Bo}. The proof of Theorem \ref{main} is based on the "polynomial method" developed by Dvir, Guth and Katz, and I do not know how -- or if -- this technique can be combined with the non-concentration hypothesis \eqref{nonConc}. 

The case $s < 1/2$ is systematically ignored in this paper, because the corresponding results in that range are quite straightforward. Finally, it seems plausible that the exponents in Theorem \ref{main} are sharp for $\delta \leq r \leq 1$, but the family of counter-examples in Section \ref{sharpness} currently requires $\delta \leq r \leq \delta^{s}$ to work.  \end{remark}

The second main result, Theorem \ref{main2} below, is directly motivated by Bourgain's proof of \cite[Theorem 3]{Bo} (which is essentially Conjecture \ref{mainConjecture} for $s$ close enough to $1/2$, and without the assumption $|E_{s}| \sim \delta^{-s}$). Here is a \emph{prestissimo} explanation of some parts of \cite{Bo}. If the result were not true, then for arbitrarily small $\epsilon,\delta > 0$, one can find a set $K$ as in Conjecture \ref{mainConjecture}, and three vectors $e_{1},e_{2},e_{3} \in S^{1}$ with separation $\sim 1$, such that $N(\pi_{e_{i}}(K),\delta) \leq \delta^{-1/2 - \epsilon}$ for $i \in \{1,2,3\}$. This counter-assumption can be used to extract strong structural information about $K$: in particular, $K$ is quantitatively \textbf{not} $1$-Ahlfors-David regular (for the definition, see Section \ref{ADR}). In the second part of the proof of \cite[Theorem 3]{Bo}, the structural information is applied to show that $K$ must, after all, have plenty of reasonably big projections. 

A major (but not the only) obstacle in applying Bourgain's method to Conjecture \ref{mainConjecture} is that the same structural conclusions cease to hold, if one replaces the assumption 
\begin{displaymath} N(\pi_{e_{i}}(K),\delta) \leq \delta^{-1/2 - \epsilon}, \qquad i \in \{1,2,3\}, \end{displaymath}
by 
\begin{displaymath} N(\pi_{e_{i}}(K),\delta) \leq \delta^{-s}, \qquad i \in \{1,2,3\}, \end{displaymath}
for some $s < 1$, possibly very close to $1$. Indeed, the $1$-dimensional four corners Cantor set $K$ is $1$-Ahlfors-David regular with very modest constants, yet it has three well-separated projections $\pi_{e_{i}}(K)$ (vertical, horizontal and $45^{\circ}$) such that $N(\pi_{e_{i}}(K),\delta) \lesssim \delta^{-s}$ with $s = \log 3/\log 4 < 1$. 

So, three directions are not enough, but how about a million? More precisely: fix $s < 1$, and assume that $N(\pi_{e}(K),\delta) \leq \delta^{-s}$ for, say, $p(s) \in \N$ well-separated vectors $e \in S^{1}$. Is it, then, true that $K$ cannot be $1$-Ahlfors-David regular with bounded constants? A positive answer to this question is the content of the second main theorem:
\begin{thm}\label{main2} Given $1/2 \leq s < 1$ and $A > 0$, there are numbers $p = p(s,A) \in \N$ and $\delta(A,s) > 0$ with the following property. Let
\begin{displaymath} S_{p} := \{e^{2\pi i k/p} : 0 \leq k < p\} \subset S^{1}, \end{displaymath}
and let $\emptyset \neq K \subset B(0,1)$ be a $1$-Ahlfors-David regular set with $\calH^{1}(K) \sim 1$ and regularity constant at most $A$. Then
\begin{displaymath} \frac{1}{p} \sum_{p \in S_{p}} N(\pi_{e}(K),\delta) \geq \delta^{-s}, \qquad 0 < \delta \leq \delta(A,s). \end{displaymath}
In particular, $\overline{\dim}_{\textup{B}} \pi_{e}(K) \geq s$ for some $e \in S_{p}$.
\end{thm}

Above, $\overline{\dim}_{\textup{B}}$ is the upper box dimension, defined for bounded sets $A \subset \R^{d}$ by
\begin{displaymath} \overline{\dim}_{\textup{B}} A := \limsup_{\delta \to 0} \frac{\log N(A,\delta)}{-\log \delta}. \end{displaymath}

\begin{remark} The precise form of the vectors in $S_{p}$ is not too important for the argument: it is only needed that, for some weights $w_{e} \sim 1/p = |S_{p}|^{-1}$, the difference
\begin{displaymath} \left| \sum_{e \in S_{p}} w_{e} \cdot f(e) - \int_{S^{1}} f(\xi) \, d\sigma(\xi) \right| \end{displaymath} 
can be made arbitrarily small for all functions $f$ on $S^{1}$ with a reasonable modulus of continuity, depending on $A$ and $s$. A more general statement would also be more awkward to write down, however, so I chose not to pursue the topic. 

Another point is that there is no analogue of Theorem \ref{main2} for Hausdorff dimension. Indeed, given any countable collection of vectors $E \subset S^{1}$, it is straightforward to construct a $1$-Ahfors-David regular set such that $\dim \pi_{e}(K) = 0$ for all $e \in E$ (and indeed for all $e \in G$, where $G \supset E$ is a suitable $G_{\delta}$-set). For the details, see \cite[Theorem 1.5]{O}. 

It is a somewhat less trivial question, whether Ahlfors-David regularity is, in fact, necessary for Theorem \ref{main2}. For instance: given $s < 1$, is it possible to find a finite set $E_{s} \subset S^{1}$ such that $\overline{\dim}_{\textup{B}} \pi_{e}(K) \geq s$ for some $e \in E_{s}$, whenever $K \subset B(0,1)$ is a compact set with $\calH^{1}_{\infty}(K) \sim 1$? Most likely, the answer is negative. Given any $\epsilon > 0$ and any finite set $D \subset \R$, an example of B. Green -- which appears in \cite[Remark 2]{KL} -- can be modified to produce a finite set $A \subset \R$ with the property that $|A + tA| \leq |A|^{1 + \epsilon}$ for all $t \in D$. Then, it seems probable that a self-similar construction with $|A|^{2}$ homotheties (mapping $0$ to the points in $A \times A$, with contraction ratios $1/|A|^{2}$) produces a set $K \subset \R^{2}$ with $0 < \calH^{1}(K) < 1$, or at least $\dim K = 1$, such that $\overline{\dim}_{\textup{B}} \pi_{t}(K) \leq 1/2 + \epsilon$ for $t \in D$. Here $\pi_{t}(x,y) = x + ty$. 
\end{remark}

The rest of the paper is organised as follows. Section \ref{sharpness} discusses the sharpness of the bound in Theorem \ref{main}. Section \ref{Kaufman} reviews some basic concepts used in the proof of Theorem \ref{main}, and gives a quick -- and well-known -- argument for the discrete Kaufman bound \eqref{discreteKaufman}. The proof of Theorem \ref{main} is given in Section \ref{mainProof}, and Section \ref{ADR} contains the proof of Theorem \ref{main2}.

Some notational remarks: $B(x,r)$ stands for a closed ball of radius $r > 0$ and centre $x \in \R^{2}$. The side-length of a cube $Q \subset \R^{d}$ is denoted by $\ell(Q)$. The inequality $A \lesssim B$ means that $A \leq CB$ for an absolute constant $C > 0$; the two-sided inequality $A \lesssim B \lesssim A$ is abbreviated to $A \sim B$. As mentioned above, $A \lessapprox B$ means that $A \lesssim \log(1/\delta)B$. The Hausdorff measure of dimension $s$ is denoted by $\calH^{s}$, and Hausdorff content by $\calH^{s}_{\infty}$. Thus
\begin{displaymath} \calH^{s}_{\infty}(A) := \inf\left\{\sum_{i = 1}^{\infty} \diam(E_{i})^{s} : A \subset \bigcup_{i = 1}^{\infty} E_{i} \right\}. \end{displaymath}
For information about Hausdorff dimension or measures, upped box dimension, or any other geometric measure theoretic concept in the text, see Mattila's book \cite{Mat}. 

\section{Acknowledgements} 

Most of this research was conducted while I was visiting Prof. D. Preiss at the University of Warwick, and I am thankful for his hospitality. I also wish to thank A. M\'ath\'e for several discussions around the topics of the paper. The proof of the second main theorem was essentially completed while I was visiting M. Hochman at the Hebrew University of Jerusalem. I am grateful for his hospitality and many useful discussions, and in particular for pointing out the existence of Lemma \ref{multiScale}. 

\section{Some worst-case examples}\label{sharpness}

Fix $\delta > 0$, $1/2 \leq s < 1$ and $\delta \leq r \leq \delta^{s}$. The example showing the sharpness of Theorem \ref{main} with these parameters can be seen in Figure \ref{fig1}. 
\begin{figure}[h!]
\begin{center}
\includegraphics[scale = 0.5]{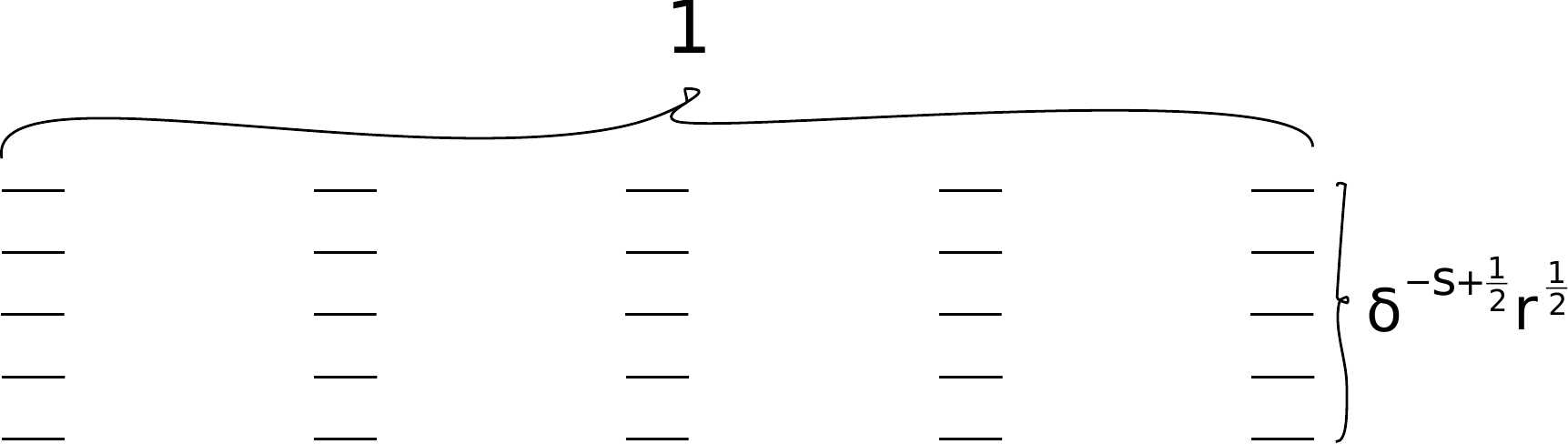}
\caption{The set in the figure contains $\delta^{-s - 1/2}r^{1/2}$ line segments of length $\delta^{s + 1/2}r^{-1/2}$ in a vertically squashed grid formation.}\label{fig1}
\end{center}
\end{figure}
To define the set precisely, let $m := \delta^{s - 1/2}r^{-1/2}$ and $n := \delta^{-2s}r$, and assume for convenience that these numbers are integers. Let 
\begin{displaymath} G := G_{1} \times G_{2} := \{k/m : 0 \leq  k \leq m - 1\} \times \{l/(mn) : 0 \leq l \leq n - 1\}. \end{displaymath}
The set $K$ is defined by $K := G + L$, where $L$ is the line segment $[0,h]$,
\begin{displaymath} h := \delta^{s}\left(\frac{\delta}{r}\right)^{1/2}. \end{displaymath} 

The first claim is that $\mathcal{H}^{1}_{\infty}(K) \sim 1$. Note that the gap of the "vertical" arithmetic progression $G_{2}$ is
\begin{displaymath} \frac{1}{mn} = \frac{1}{\delta^{s - 1/2}r^{-1/2}\delta^{-2s}r} = \delta^{s}\left(\frac{\delta}{r}\right)^{1/2} = h, \end{displaymath}
which equals the length of the line segments. Moreover, the gap of the "horizontal" progression $G_{1}$ is 
\begin{displaymath} \frac{1}{m} = \frac{n}{mn} \sim \diam(G_{2}). \end{displaymath}
In English, if the "vertical stacks" are rotated by $90$ degrees, they roughly fit inside the gaps of the progression $G_{2}$. It follows easily from these facts that the measure $\mathcal{H}^{1}|_{K}$ satisfies the Frostman inequality $\mathcal{H}^{1}|_{K}(B(x,r)) \lesssim r$ with absolute constants (in fact $K$ is even $1$-AD-regular), and hence $\mathcal{H}^{1}_{\infty}(K) \gtrsim 1$. 

To understand the projections of $K$, one needs the following easy lemma:
\begin{lemma}\label{gridProjections} Let that $A_{1},A_{2} \subset \R$ be arithmetic progressions and $e \in S^{1}$. If 
\begin{displaymath} |\pi_{e}^{-1}\{0\} \cap (A_{1} \times A_{2})| \geq p, \end{displaymath}
then $|\pi_{e}(A_{1} \times A_{2})| \leq 4|A_{1}||A_{2}|/p$.
\end{lemma}

\begin{proof} The hypothesis means that $\pi_{e}(x,y) = 0$ has at least $p$ solutions $(x_{1},y_{2}),\ldots,(x_{p},y_{p}) \in A_{1} \times A_{2}$. Now, fix any $t \in \pi_{e}(A_{1} \times A_{2})$. Then $\pi_{e}(x^{t},y^{t}) = t$ for some $(x^{t},y^{t}) \in A_{1} \times A_{2}$, but also
\begin{displaymath} \pi_{e}(x^{t} + x_{j},y^{t} + y_{j}) = t, \qquad 1 \leq j \leq p.  \end{displaymath}
The points $x^{t} + x_{j}$ are contained in $A_{1}' := A_{1} + A_{1}$, and the points $y^{t} + y_{j}$ are contained in $A_{2}' := A_{2} + A_{2}$. Thus, for every $t \in \pi_{e}(A_{1} \times A_{2})$, the intersection $(A_{1}' \times A_{2}') \cap \pi_{e}^{-1}\{t\}$ contains at least $p$ elements. This gives
\begin{displaymath} |\pi_{e}(A_{1} \times A_{2})| \leq \frac{|A_{1}' \times A_{2}'|}{p} \leq \frac{4|A_{1}||A_{2}|}{p}, \end{displaymath}
as claimed. \end{proof}

Now, it is time to define a set of slopes $e \in S^{1}$ such that $\pi_{e}(G)$ -- and also $\pi_{e}(K)$ -- is small. For $1 \leq k \leq \delta^{s}/r$, consider the line spanned by the origin and any point of the form $(k/m,l/(mn)) \in G$ with $0 \leq l \leq k \cdot \delta^{-3s + 1/2}r^{3/2}$. The slopes of such lines form the set
\begin{displaymath} S := \left\{\frac{l}{kn} : 1 \leq k \leq \delta^{s}/r \text{ and } 0 \leq l \leq k \cdot \delta^{-3s + 1/2}r^{3/2}\right\}. \end{displaymath} 
It will eventually be shown that $\pi_{e}(K)$ is small for every $e$ perpendicular to a line with slope in $S$, but one first needs to analyse $S$ a bit. First, assume that two elements of $S$ are closer than $r$ apart, say 
\begin{displaymath} \left|\frac{l_{1}}{k_{1}n} - \frac{l_{2}}{k_{2}n} \right| < r. \end{displaymath}
Then, recalling that $n = \delta^{-2s}r$, and that $k_{1},k_{2} \leq \delta^{s}/r$,
\begin{displaymath} |l_{1}k_{2} - l_{2}k_{1}| < r nk_{1}k_{2} \leq r\delta^{-2s} r \delta^{2s}/r^{2} = 1, \end{displaymath}
which forces $l_{1}k_{2} - l_{2}k_{1} = 0$, since $l_{1}k_{2} - l_{2}k_{1} \in \Z$. In other words,
\begin{equation}\label{form16}  \left|\frac{l_{1}}{k_{1}n} - \frac{l_{2}}{k_{2}n} \right| < r \quad \Longrightarrow \quad \frac{l_{1}}{k_{1}n} = \frac{l_{2}}{k_{2}n}, \end{equation}
and hence the slopes in $S$ are $r$-separated. 

The next observation is that $|S| \gtrsim \delta^{-s}(\delta/r)^{1/2}$ (which is incidentally the bound from Theorem \ref{main}). To prove this, fix a small constant $c > 0$, and let $\calS_{c}$ consist of those pairs $(k,l) \in \N \times \N$ such that
\begin{equation}\label{form17} 1 \leq k \leq \frac{\delta^{s}}{r} \quad \text{and} \quad 0 \leq l \leq k \cdot \delta^{-3s + 1/2}r^{3/2}, \end{equation}
and $l/k \neq l'/k'$ for any $1 \leq |k'| \leq c\delta^{s}/r$ and $|l'| \leq 2\delta^{-2s + 1/2}r^{1/2}$ (the letters $k'$ and $l'$ also stand for integers). So, roughly speaking, one is considering those slopes in $S$, which cannot (also) be expressed as slopes $l'/(k'n)$ with a small denominator $k'$. Then $|\calS_{c}| \gtrsim \delta^{-s}(\delta/r)^{1/2}$, if $c > 0$ is small enough, since the total number of pairs $(k,l)$ respecting \eqref{form17} is $\gtrsim \delta^{-s}(\delta/r)^{1/2}$, whereas the number of those satisfying $0 < |k'| \leq c\delta^{s}/r$ and $|l'| \leq 2\delta^{-2s + 1/2}r^{1/2}$ is only $\lesssim c\delta^{-s}(\delta/r)^{1/2}$. 

It will now be shown that the mapping $(k,l) \mapsto l/k$ restricted to $\calS_{c}$ is $C$-to-$1$ for some absolute $C \geq 1$ (depending on $c$): this will of course prove that $|S| \gtrsim \delta^{-s}(\delta/r)^{1/2}$, as desired. Assume that $(k_{1},l_{1}),\ldots,(k_{C},l_{C}) \in \calS_{c}$ are distinct pairs satisfying $l_{i}/k_{i} = l_{j}/k_{j}$. Then also
\begin{equation}\label{form18} \frac{l_{i}}{k_{i}} = \frac{l_{i} - l_{j}}{k_{i} - k_{j}} = \frac{l_{j}}{k_{j}}, \qquad 1 \leq i,j \leq C. \end{equation}
Since $|k_{i} - k_{j}| \leq 2\delta^{-s}/r$, one can pigeonhole a pair of pairs $(k_{i},l_{i}),(k_{j},l_{j})$ with $0 < |k_{i} - k_{j}| \lesssim 2\delta^{-s}/(Cr)$, and in particular $1 \leq |k_{i} - k_{j}| \leq c\delta^{s}/r$ for large enough $C$. Since also $|l_{i} - l_{j}| \leq 2\delta^{-2s + 1/2}r^{1/2}$ for any pair of indices $i,j$, one sees from \eqref{form18} and the definition of $\calS_{c}$ that in fact neither of the pairs $(k_{i},l_{i})$ and $(k_{j},l_{j})$ can lie in $\calS_{c}$. This gives an upper bound for $C$, and the proof of $|S| \gtrsim \delta^{-s}(\delta/r)^{1/2}$ is complete.

To sum up the progress so far, one has found a set $K$ with $\calH_{\infty}^{1}(K) \sim 1$, and a set of $r$-separated slopes $S$ with cardinality $|S| \gtrsim \delta^{-s}(\delta/r)^{1/2}$. It remains to prove that whenever $e \in S^{1}$ is perpendicular to a line with slope in $S$, the projection $\pi_{e}(K)$ can be covered by $\lesssim \delta^{-s}$ intervals of length $\delta$. This uses Lemma \ref{gridProjections}, as the plan is to prove first that $|\pi_{e}(G)| \lesssim \delta^{-s}$. Fix a slope $l/(kn) \in S$, with $1 \leq k \leq \delta^{s}/r$ and $0 \leq l \leq k \cdot \delta^{-3s + 1/2}r^{3/2}$, and consider the line $\ell$ passing the origin and $(k/m,l/(mn)) \in G$. Then, for $j \in \N$, one has
\begin{displaymath} \left(\frac{jk}{m},\frac{jl}{mn} \right) \in G \cap \ell, \end{displaymath}
as long as $jk \leq m - 1 = \delta^{s - 1/2}r^{-1/2} - 1$ and $jl \leq n - 1 = \delta^{-2s}r - 1$. One checks from the restraints on $k$ and $l$ in the definition of $S$ that this holds if $j \leq (r/\delta)^{1/2}/2$. Consequently, $|\ell \cap G| \gtrsim (r/\delta)^{1/2}$ and hence, if $e$ is perpendicular to $\ell$ (then $\ell = \pi^{-1}_{e}\{0\}$), Lemma \ref{gridProjections} tells us that
\begin{displaymath} |\pi_{e}(G)| \lesssim \frac{|G|}{(r/\delta)^{1/2}} = \frac{\delta^{s - 1/2}r^{-1/2}\delta^{-2s}r}{(r/\delta)^{1/2}} = \delta^{-s}. \end{displaymath}

Now, the very final step is to check that $\pi_{e}(K) = \pi_{e}(G + L)$ is contained in the $\sim \delta$-neighbourhood fo $\pi_{e}(G)$, and this follows from the fact that $\pi_{e}(L)$ is an interval of length $\lesssim \delta$, whenever $e$ is perpendicular to a line with slope in $S$. Simply observe that the slopes in $S$ satisfy 
\begin{displaymath} \frac{l}{kn} \leq \frac{\delta^{-3s + 1/2}r^{3/2}}{\delta^{-2s}r} = \delta^{-s + 1/2}r^{1/2}, \end{displaymath}
and recall that $L$ is a horizontal line segment of length $h = \delta^{s}(\delta/r)^{1/2}$. By elementary trigonometry, the length of $\pi_{e}(L)$ is roughly $l/(kn) \cdot h \leq \delta$. 

\section{Basic concepts and Kaufman's bound}\label{Kaufman}

The proof of the upper bound in Theorem \ref{main} begins in this section. Fix the parameters $\delta,r,s$, and choose $\tau \in [0,1]$ so that $r = \delta^{\tau}$. The task is to estimate $N(E_{s},r) = N(E_{s},\delta^{\tau})$ from above, which is equivalent to bounding the cardinality of a maximal $\delta^{\tau}$-separated subset of $E_{s}$ from above. With this in mind, and from this point on, assume that $E_{s}$ is a $\delta^{\tau}$-separated subset of $\{e \in S^{1} : N(\pi_{e}(K),\delta) \leq \delta^{-s}\}$.

It is also convenient to discretise the set $K$ at the scale $\delta$. The following definition is essentially due to Katz and Tao \cite{KT}:

\begin{definition}[$(\delta,1)$-sets] A finite set $P \subset \R^{2}$ is called a $(\delta,1)$-set, if $P$ is $\delta$-separated, and
\begin{displaymath}  |P \cap B(x,r)| \lesssim \frac{r}{\delta}, \qquad x \in \R^{2}, \: r \geq \delta. \end{displaymath}
Here $|\cdot |$ means cardinality.
\end{definition}

\begin{lemma}\label{discretisation} Let $\delta > 0$, and let $K \subset \R^{2}$ be a set with $\calH^{1}_{\infty}(K) =: \kappa > 0$. Then, there exists a $(\delta,1)$-set $P \subset K$ with $|P| \gtrsim \kappa \cdot \delta^{-1}$.
\end{lemma}

\begin{proof} Choose a $\delta$-net inside $K$ and discard surplus points. For more details, see \cite[Proposition A.1]{FO}.
\end{proof}

\begin{definition}[Incidences] Let $\calT$ be a family of infinite tubes of width $\delta$, and let $P \subset \R^{2}$ be a finite set of points. The set of \emph{incidences} $I(P,\calT)$ between $P$ and $\calT$ is the following family of pairs:
\begin{displaymath} I(P,\calT) := \{(p,T) : p \in P, \: T \in \calT \text{ and } p \in T\}. \end{displaymath}
\end{definition}

The definition will be applied to subsets of the set $P$ from Lemma \ref{discretisation}, and subsets of the following family $\calT$ of tubes:
\begin{definition}[Tubes $\calT$]\label{tubes} Let $P \subset K$. For each $e \in E_{s}$, cover $\pi_{e}(P)$ by $\leq \delta^{-s}$ intervals $I$ of length $\delta$ and bounded overlap (this is possible since $N(\pi_{e}(P),\delta) \leq N(\pi_{e}(K),\delta) \leq \delta^{-s}$), and let $\calT_{e}$ be the family of $\delta$-tubes of the form $\pi_{e}^{-1}(I)$. Then, let
\begin{displaymath} \calT := \bigcup_{e \in E_{s}} \calT_{e}. \end{displaymath}
\end{definition}

The basic strategy in the proofs will be to bound $|I(P,\calT)|$ both from above and below. The desirable lower bound is trivial:
\begin{lemma}\label{lowerBound} Let $P$ be an arbitrary finite set in $\R^{2}$, and construct $\calT$ as in Definition \ref{tubes}. Then
\begin{displaymath} |I(P,\calT)| \geq |P||E_{s}|. \end{displaymath}
\end{lemma}

\begin{proof} Each point $p \in P$ is contained in at least one tube from each family $\calT_{e}$, $e \in E_{s}$.
\end{proof}

Kaufman's $\delta^{-s}$-bound \eqref{discreteKaufman} will follow from comparing the previous bound with the one provided by the next proposition.
\begin{proposition}\label{cellIncidences} Assume that $P \subset K$ is a $(\delta,1)$-set, and $\calT$ is the collection of tubes from Definition \ref{tubes}, associated with $P$. Then
\begin{displaymath} |I(P,\calT')| \lessapprox |P||\calT'|^{1/2} + |\calT'| + \sqrt{\delta^{-\tau}|P||\calT'|} \end{displaymath}
for any subset $\calT' \subset \calT$. 
\end{proposition}

\begin{proof} Using the definition of $I(P,\calT')$ and Cauchy-Schwarz,
\begin{equation}\label{KaufmanEstimate} |I(P,\calT')| = \sum_{T \in \calT'} |\{p \in P \cap T\}| \leq |\calT'|^{1/2} \left( \sum_{T \in \calT'} |\{(p,q) : p,q \in P \cap T\}| \right)^{1/2}. \end{equation}
It remains to estimate the sum on the right hand side:
\begin{align*} \sum_{T \in \calT'} & |\{(p,q) : p,q \in P \cap T\}| = \sum_{p,q \in P} |\{T \in \calT' : p,q \in P \cap T\}|\\
& = \sum_{p \in P} |\{T \in \calT' : p \in P \cap T\}| + \sum_{p \neq q} |\{T \in \calT' : p,q \in P \cap T\}|. \end{align*}
The first sum equals $|I(P,\calT')|$ again, which gives rise to the $|\calT'|$-term in \eqref{KaufmanEstimate}. To estimate the second sum, one uses the finite overlap of the tubes in any fixed family $\calT_{e}$ to estimate
\begin{displaymath} \sum_{p \neq q} |\{T \in \calT' : p,q \in P \cap T\}| \lesssim \sum_{p \neq q} |\{e \in E_{s} : p,q \in T \text{ for some } T \in \calT_{e}\}|. \end{displaymath}
At this point, one applies the standard geometric fact that the set of vectors $e \in S^{1}$ such that $p,q$ can share a common $\delta$-tube in $\calT_{e}$ is contained in two arcs of length $\lesssim \delta/|p - q|$. Since the vectors in $E_{s}$ are $\delta^{\tau}$-separated, this leads to
\begin{equation}\label{form1} |\{e \in E_{s} : p,q \in T \text{ for some } T \in \calT_{e}\}| \lesssim \max\left\{\frac{\delta^{1 - \tau}}{|p - q|}, 1\right\}. \end{equation}
Observe that the "$1$" is really needed here, because if $|p - q|$ is far greater than $\delta^{1 - \tau}$, the arcs mentioned above have length far smaller than $\delta^{\tau}$, but it is still perfectly possible for one $\delta^{\tau}$-separated vector to land in any such arc. The bound \eqref{form1} leads to
\begin{align*} \sum_{p \neq q} |\{e \in E_{s} \colon & p,q \in T \text{ for some } T \in \calT_{e}\}| \lesssim \sum_{p \neq q} \max\left\{\frac{\delta^{1 - \tau}}{|p - q|},1\right\}\\
& = \sum_{p \in P} \left( \sum_{q : |p - q| \leq \delta^{1 - \tau}} \frac{\delta^{1 - \tau}}{|p - q|} + \sum_{q : |p - q| > \delta^{1 - \tau}} 1 \right)\\
& \lesssim \sum_{p \in P} \left( \delta^{-\tau} \log\left(\frac{1}{\delta}\right) + |P| \right) \lessapprox \delta^{-\tau}|P| + |P|^{2}. \end{align*} 
The inequality between the last two lines was obtained by splitting $P$ around $p$ in annuli of radius $\sim 2^{-j}$, $0 \leq j \leq \log(1/\delta)$, and using the $(\delta,1)$-set hypothesis. Rearranging terms completes the proof.
\end{proof}

To prove Kaufman's $\delta^{-s}$-bound \eqref{discreteKaufman} (or the $r = \delta$ case of Theorem \ref{main}), one uses Lemma \ref{discretisation} to find a $(\delta,1)$-set $P \subset K$ with $|P| \sim \delta^{-1}$. Then, the lower and upper bounds of Lemma \ref{lowerBound} and Proposition \ref{cellIncidences} (with $\calT' = \calT$) combined yield
\begin{displaymath} \delta^{-1}|E_{s}| \lesssim |P||E_{s}| \leq |I(P,\calT)| \lessapprox \delta^{-1}|\calT|^{1/2} + |\calT|. \end{displaymath}
Since $|\calT_{e}| \leq \delta^{-s}$ for every $e \in E_{s}$, this gives
\begin{displaymath} |E_{s}| \lessapprox (\delta^{-s}|E_{s}|)^{1/2} + |E_{s}|\delta^{1 - s}. \end{displaymath}
Given that $s < 1$, the term $|E_{s}|\delta^{1 - s}$ cannot dominate the left hand side, and the proof is finished by taking squares and moving terms. 

In the proof of Theorem \ref{main}, one has to make more efficient use of Proposition \ref{cellIncidences}: the key point is that it gives a reasonably good bound for $|I(P,\calT)|$, when $|P| \approx \delta^{-\tau}$ -- which is crucially better than the best possible bound obtainable with mere $\delta$-separation. So, the strategy will be to use an algebraic variety -- a zero-set of a polynomial in two variables -- to partition $P$ into chunks of approximately this size, and then control the incidences in each chunk separately. As is common with such a cell-decomposition argument, one has to handle separately the case where most of $P$ is concentrated in the $\delta$-neighbourhood of the variety.

\section{Proof of the first main theorem}\label{mainProof}

A central tool is the polynomial cell decomposition theorem of Guth and Katz, see \cite[Theorem 4.1]{GK}, which is quoted below:
\begin{thm}[Guth-Katz] Let $P \subset \R^{2}$ be a finite set of points, and let $D \geq 1$ be an integer. Then, there exists an algebraic variety $Z$ of degree $\deg(Z) \leq D$ with the following property: the complement $\R^{2} \setminus Z$ is the union of $\leq D^{2}$ open cells $O_{i}$ such that $\partial O_{i} \subset Z$, and and each cell contains $\lesssim |P|/D^{2}$ points of $P$. 
\end{thm}

To begin the proof of Theorem \ref{main} in earnest, apply the partitioning theorem with the $(\delta,1)$-set $P \subset K$ of cardinality $|P| \sim \delta^{-1}$, obtained from Lemma \ref{discretisation}, and with some large integer $D \geq 1$ to be optimised later. Let $Z$ be the ensuing polynomial surface of degree $\leq D$, and let $\tilde{O}_{i}$, $1 \leq i \leq N \leq D^{2}$ be the components of the complement $\R^{2} \setminus Z$. Finally, let
\begin{displaymath} O_{i} := \tilde{O}_{i} \setminus Z(\delta), \end{displaymath}
where $Z(\delta) := \{x : \dist(x,Z) \leq \delta\}$ is the closed $\delta$-neighbourhood of $Z$. The reason for defining the cells $O_{i}$ so is the following simple consequence of B\'ezout's theorem (first observed in \cite{Gu}):
\begin{lemma}\label{guthLemma} Let $T$ be an infinite tube of width $\delta$. Then $T$ can intersect at most $D + 1$ cells $O_{i}$.
\end{lemma}

\begin{proof} Let $L_{T}$ be the central line of $T$. For every $i$ such that $T \cap O_{i} \neq \emptyset$, one has $L_{T} \cap \tilde{O}_{i} \neq \emptyset$, and this is only possible for $\leq D + 1$ values of $i$: namely, if there were $D + 2$ values or more, then $L_{T}$ would contain at least $D + 1$ points on the polynomial surface $Z$, and by B\'ezout's theorem, this would force $L_{T}$ to be contained on $Z$. Consequently, $T$ would be contained in the $\delta/2$-neighbourhood of $Z$ and could not, in fact, touch any of the cells $O_{i}$. \end{proof}

The proof of Theorem \ref{main} now divides into two main cases, according to whether or not most of the points in $P$ are contained in the union of the cells $O_{i}$. The argument in the first, "cellular" case closely resembles a (by now) standard proof of the Szemer\'edi-Trotter incidence theorem, while the "non-cellular" situation arguably requires more case-specific reasoning. As a final remark, the proof of Theorem \ref{main} would be shorter and require no polynomials, if the set $P$ had a product form, say $P = A \times A$, to begin with. Then one could perform the cell-decomposition by hand using two perpendicular families of straight lines, and the "non-cellular" case could not even occur. 
\begin{figure}[h!]
\begin{center}
\includegraphics[scale = 0.6]{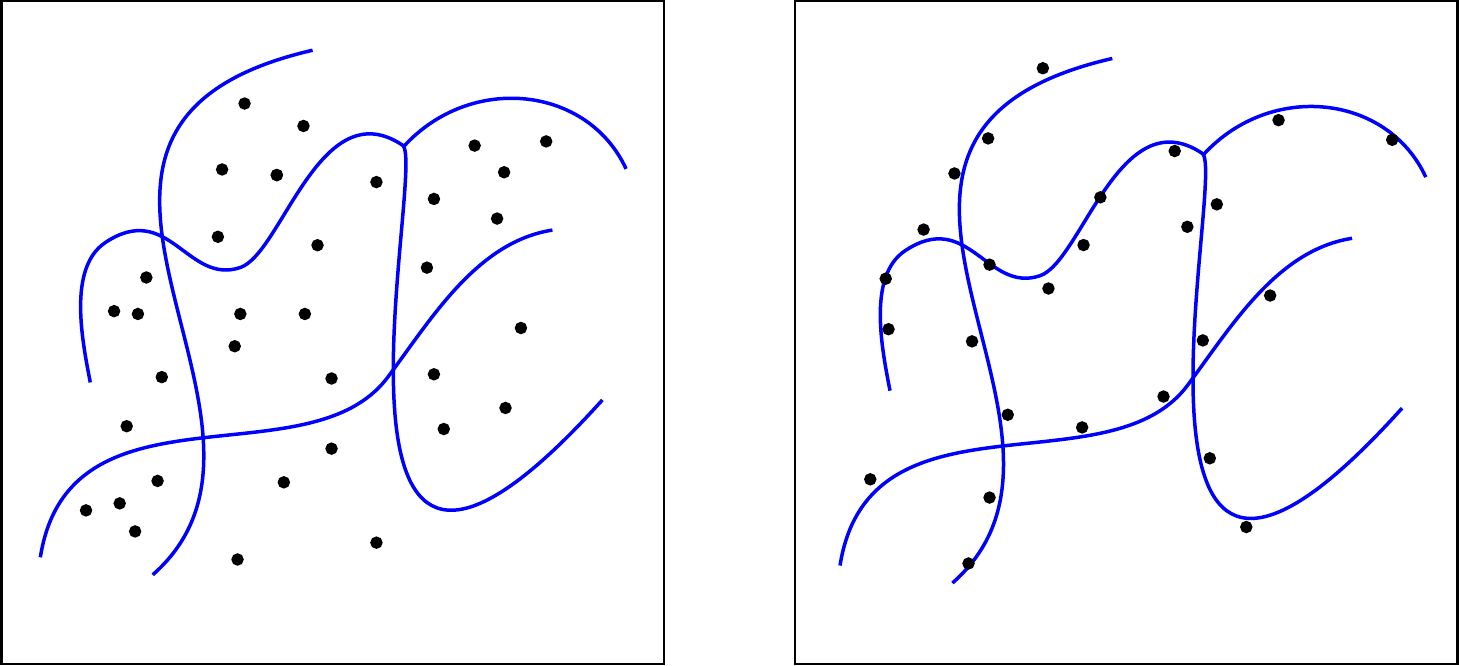}
\caption{The cellular and non-cellular cases}
\end{center}
\end{figure}

\subsection{The cellular case} In this subsection, assume that $|\tilde{P}| \geq |P|/2 \sim \delta^{-1}$, where 
\begin{equation}\label{form2} \tilde{P} := P \cap \bigcup_{i = 1}^{N} O_{i}. \end{equation}

First, discard all the cells, and the points of $\tilde{P}$ within, such that $|\tilde{P} \cap O_{i}| < \delta^{-\tau}$. Since the number of cells is bounded by $D^{2}$, this results in the removal of at most $D^{2}\delta^{-\tau}$ points of $\tilde{P}$, and this is smaller than $|\tilde{P}|/2$ as long as
\begin{equation}\label{form3} D \leq c\delta^{(\tau - 1)/2} \end{equation}
for some small absolute constant $c$. Assume \eqref{form3} in the sequel, and note that the remaining points of $\tilde{P}$ still satisfy \eqref{form2}: hence, keep the notation $\tilde{P}$ for convenience, and observe that $|O_{i} \cap \tilde{P}| \geq \delta^{-\tau}$ for the non-empty cells $O_{i}$.

Let $\calT$ be the collection of tubes introduced in Definition \eqref{tubes}, with $P$ replaced by $\tilde{P}$ (the definition of $E_{s}$ need not be changed to reflect the projections of $\tilde{P}$). Then
\begin{equation}\label{form4} |I(\tilde{P},\calT)| \geq |\tilde{P}||E_{s}| \sim \delta^{-1}|E_{s}| \end{equation}
by Lemma \ref{lowerBound}, and it remains to find an upper bound in the spirit of the end of the previous section.

First, write
\begin{equation}\label{form50} |I(\tilde{P},\calT)| = \sum_{i = 1}^{N} |I(\tilde{P} \cap O_{i},\calT)| = \sum_{i = 1}^{N} |I(\tilde{P} \cap O_{i},\calT^{i})|, \end{equation}
where $\calT^{i}$ is the collection of tubes $T \in \calT$ with $T \cap O_{i} \neq \emptyset$, and the sum only runs over the non-empty cells $O_{i}$. Observe that $\tilde{P} \cap O_{i}$ and $\calT^{i} \subset \calT$ satisfy the assumptions of Proposition \ref{cellIncidences}, so 
\begin{align*} |I(\tilde{P} \cap O_{i},\calT^{i})| & \lessapprox |\tilde{P} \cap O_{i}||\calT^{i}|^{1/2} + |\calT^{i}| + \sqrt{\delta^{-\tau}|\tilde{P} \cap O_{i}||\calT^{i}|}\\
& \lesssim |\tilde{P} \cap O_{i}||\calT^{i}|^{1/2} + |\calT^{i}|, \end{align*}
where the latter inequality used $\delta^{-\tau} \leq |\tilde{P} \cap O_{i}|$. Plugging the estimate into \eqref{form50}, recalling that $|\tilde{P} \cap O_{i}| \lesssim |\tilde{P}|/D^{2}$ and $N \leq D^{2}$, and using Cauchy-Schwarz yields
\begin{align*} |I(\tilde{P},\calT)| & \lessapprox \sum_{i = 1}^{N} |\tilde{P} \cap O_{i}||\calT^{i}|^{1/2} + \sum_{i = 1}^{N} |\calT^{i}|\\
& \lesssim \frac{|\tilde{P}|}{D} \left( \sum_{i = 1}^{N} |\calT^{i}| \right)^{1/2} + \sum_{i = 1}^{N} |\calT^{i}|. \end{align*}
Finally, by Lemma \ref{guthLemma},
\begin{displaymath} \sum_{i = 1}^{N} |\calT^{i}| = \sum_{T \in \calT} \sum_{i = 1}^{N} \chi_{\{T \cap O_{i} \neq \emptyset\}} \leq (D + 1)|\calT|, \end{displaymath}
so that
\begin{displaymath} \delta^{-1}|E_{s}| \lesssim |I(\tilde{P},\calT)| \lessapprox \frac{|\tilde{P}|}{D^{1/2}} |\calT|^{1/2} + D|\calT| \lesssim \frac{\delta^{-1}}{D^{1/2}}\left(|E_{s}|\delta^{-s}\right)^{1/2} + D|E_{s}|\delta^{-s}, \end{displaymath}
using \eqref{form4} in the left-hand side inequality. The second term on the right hand side cannot dominate the left hand side, if $D$ is significantly smaller than $\delta^{s - 1}$: tracking the constants behind the $\lessapprox$-notation, and combining with the restriction coming from \eqref{form3}, the correct thing to assume is
\begin{displaymath} D \leq c\min \{\delta^{s - 1}/\log(1/\delta),\delta^{(\tau - 1)/2}\}. \end{displaymath}
for some small absolute constant $c > 0$ (the second term in the $\min$ comes from \eqref{form3}). For such a choice of $D$,
\begin{equation}\label{form5} |E_{s}| \lessapprox \frac{\delta^{-s}}{D}. \end{equation}
This finishes the proof of the cellular case.  The degree $D$ will be optimised later.

\subsection{The non-cellular case} In this subsection, assume that $|\tilde{P}| \geq |P|/2 \sim \delta^{-1}$, where
\begin{displaymath} \tilde{P} := P \cap Z(\delta). \end{displaymath}
The strategy is to use the existence of many small projections to force $Z$ to contain many lines, which is impossible if $D$ is small enough.

Since every point in $p \in \tilde{P}$ lies in the $\delta$-neighbourhood of $Z$, there exists a point $z_{p} \in Z$ with $|p - z_{p}| \leq \delta$. Let $C_{p}$ be the component of $Z$ containing $z_{p}$. Given a number $\nu > 0$ to be specified momentarily, call $p$ an $\nu$-\emph{bad} point, if there exist two vectors $e_{1},e_{2} \in E_{s}$ with $|e_{1} - e_{2}| \gtrsim |E_{s}|\delta^{\tau}$ such that the maximal (component) interval of $\pi_{e_{i}}(B(p,2\delta) \cap C_{p})$ containing $\pi_{e_{i}}(z_{p})$ has length $\leq \nu$ for $i = 1,2$ (including the case where the component interval is just the single point $\pi_{e_{i}}(z_{p})$). The claim is that there cannot be many $\nu$-bad points in $\tilde{P}$. Figure \ref{fig2} is relevant to the following argument.
\begin{figure}[h!]
\begin{center}
\includegraphics[scale = 0.5]{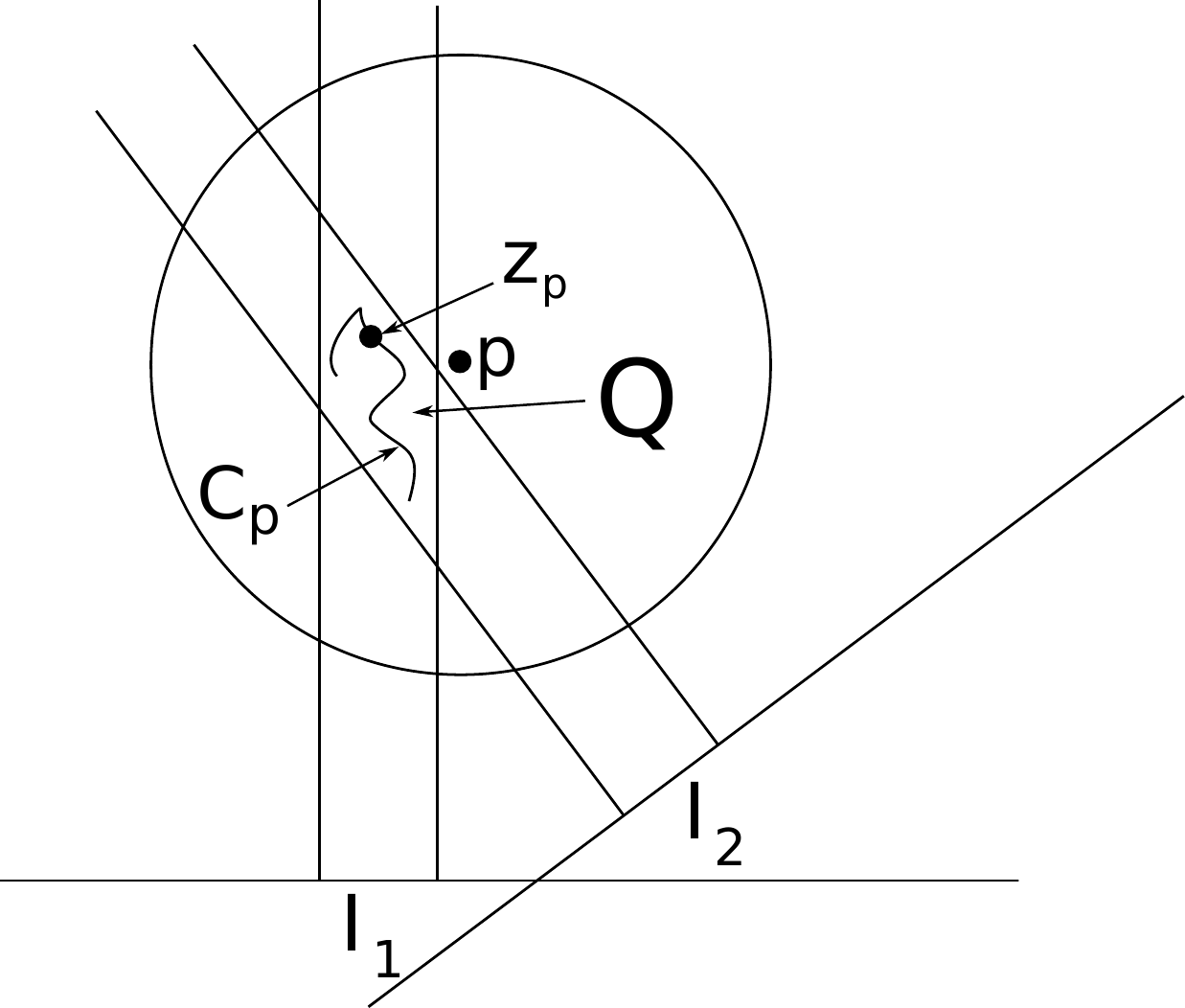}
\caption{The picture near an $\nu$-bad point $p \in \tilde{P}$.}\label{fig2}
\end{center}
\end{figure}

Fix a $\nu$-bad point $p \in \tilde{P}$, so that the corresponding component intervals of both $\Pi_{1} := \pi_{e_{1}}(B(p,2\delta) \cap C_{p})$ and $\Pi_{2} := \pi_{e_{2}}(B(p,2\delta) \cap C_{p})$ have length $\leq \nu$. Then, one can find two open intervals $I_{1}$ and $I_{2}$, containing $\pi_{e_{1}}(z_{p})$ and $\pi_{e_{2}}(z_{p})$, respectively, of length $\leq 2\nu$, and such that 
\begin{equation}\label{form7} \partial I_{1} \cap \Pi_{1} = \emptyset = \partial I_{2} \cap \Pi_{2}. \end{equation} 
By elementary geometry, the box $Q := \pi_{e_{1}}^{-1}(I_{1}) \cap \pi_{e_{2}}^{-1}(I_{2})$ has diameter
\begin{displaymath} \diam(Q) \lesssim \frac{\nu}{|E_{s}|\delta^{\tau} + \nu}. \end{displaymath}
Hence, $Q$ is an open box containing $z_{p}$, and contained in $B(p,2\delta)$ if $\nu \leq c|E_{s}|\delta^{1 + \tau}$ for a sufficiently small constant $c > 0$ (recall that $z_{p} \in B(p,\delta)$). It follows from these observations that, for such $\nu > 0$, in fact $C_{p} \subset Q \subset B(p,2\delta)$: otherwise $C_{p}$ should intersect the boundary of $Q$, and since this happens inside $B(p,2\delta)$, one has either $\partial I_{1} \cap \Pi_{1} \neq \emptyset$ or $\partial I_{2} \cap \Pi_{2} \neq \emptyset$ contrary to \eqref{form7}.

To summarise, if 
\begin{equation}\label{rConstraint} \nu = c|E_{s}|\delta^{1 + \tau} \end{equation}
for a suitable small constant $c > 0$, then for every $\nu$-bad point $p \in \tilde{P}$, there exists a component of $Z$ inside $B(p,2\delta)$. By Harnack's curve theorem, see \cite{Ha}, the number of components of $Z$ is bounded by $\lesssim D^{2}$, so as long as \eqref{rConstraint} holds, and 
\begin{equation}\label{DRestriction} D^{2} \leq c|\tilde{P}|, \end{equation}
there are at most $|\tilde{P}|/2$ $\nu$-bad points in $\tilde{P}$. These points are now discarded from $\tilde{P}$. For notational convenience, the remaining points are still denoted by $\tilde{P}$.

Assuming that $|E_{s}| \geq 2$ -- as one may -- pick two vectors $e_{1},e_{2} \in E_{s}$ with $|e_{1} - e_{2}| \gtrsim |E_{s}|\delta^{\tau}$. Since no point in $\tilde{P}$ is $\nu$-bad, the following holds for either $i = 1$ or $i = 2$: there is a subset $P' \subset \tilde{P}$ of cardinality $|P'| \geq |\tilde{P}|/2$ such that $\calH^{1}(\pi_{e_{i}}(B(p,2\delta) \cap Z)) \geq \nu$ for all $p \in P'$.\footnote{Obviously this is a weaker requirement than the existence of a long component interval in $\pi_{e_{i}}(B(p,2\delta) \cap C_{p})$: the stronger claim was simply introduced, because it was easier to prove.} Assume that this holds for $i = 1$. 

Next, observe that, if $c > 0$ is small enough, there exists a tube of the form $T_{0} := \pi_{e_{1}}^{-1}(I_{0})$, with $\ell(I_{0}) = \delta$, and $|P' \cap T_{0}| \geq c\delta^{s - 1}$. Indeed, since $e_{1} \in E_{s}$, one can first cover $\pi_{e_{1}}(P') \subset \pi_{e_{1}}(P)$ with $\leq \delta^{-s}$ intervals $I$ of length $\delta$, and then observe that only $|P'|/2$ points can be contained in tubes of the form $\pi_{e_{1}}^{-1}(I)$ with $|P' \cap \pi_{e_{1}}^{-1}(I)| < c\delta^{s - 1}$: in particular, there exists a tube $T_{0}$ satisfying the opposite inequality. Finally, assume that the points $p \in P' \cap T_{0}$ are $5\delta$-separated (if not, discard additional points and observe that $\gtrsim \delta^{s - 1}$ points in $P' \cap T_{0}$ remain).

Pick a line $l$ passing through -- and parallel to -- $5T_{0}$ uniformly at random, and for a given point $p \in P' \cap T_{0}$, consider the random variable 
\begin{displaymath} X_{p}(l) := \chi_{\{l \cap Z \cap B(p,2\delta) \neq \emptyset\}}. \end{displaymath}
Since $\calH^{1}(B(p,2\delta) \cap Z) \geq \nu = c|E_{s}|\delta^{1 + \tau}$, and $\pi_{e_{1}}(B(p,2\delta)) \subset 5I_{0}$, one has $\E[X_{p}] \gtrsim |E_{s}|\delta^{\tau}$, and
\begin{displaymath} \E\left[\sum_{p \in P' \cap T_{0}} X_{p}\right] \gtrsim \delta^{s  + \tau - 1}|E_{s}|. \end{displaymath}
Since the points $p \in P' \cap T_{0}$ are $5\delta$-separated, the sum $\sum X_{p}(l)$ gives a lower bound for \textbf{distinct} intersections of $l$ with $Z$. On the other hand, it follows from B\'ezout's theorem that almost every line in any fixed direction hits $Z$ in at most $D$ distinct points (since the lines with more than $D$ intersections are contained in $Z$, and $Z$ has null Lebesgue measure), so
\begin{displaymath} |E_{s}| \lesssim D\delta^{1 - s - \tau}. \end{displaymath}

The only restriction on $D$ required for this inequality was $D^{2} \leq c\delta^{-1}$, recalling \eqref{DRestriction}.  

\subsection{Conclusion of the proof} With $r = \delta^{\tau}$, the claim was that 
\begin{equation}\label{form10} N(E_{s},\delta^{\tau}) \lessapprox \min\{\delta^{-s + (1 - \tau)/2},\delta^{-\tau}\}. \end{equation}
Obviously, 
\begin{displaymath} N(E_{s},\delta^{\tau}) \lesssim \delta^{-\tau}, \end{displaymath}
and this coincides with the minimum in \eqref{form10}, if $\tau \leq 2s - 1$. So, one may assume that $\tau > 2s - 1 \geq 0$. Now, the previous two subsections have shown that 
\begin{displaymath} |E_{s}| \lessapprox \max\left\{ \frac{\delta^{-s}}{D},D\delta^{1 - s - \tau}\right\}, \end{displaymath}
where $D$ is any integer satisfying
\begin{equation}\label{form8} D \leq c \min\{\delta^{s - 1}/\log(1/\delta),\delta^{(\tau - 1)/2}\}. \end{equation}
Since $\tau > 2s - 1$, one has $\delta^{(\tau - 1)/2} < \delta^{s - 1}/\log(1/\delta)$, so one is allowed to choose $D = c\delta^{(\tau - 1)/2}$, and this results in
\begin{displaymath} |E_{s}| \lessapprox \delta^{-s + (1 - \tau)/2}. \end{displaymath}
The proof of \eqref{form10}, and Theorem \ref{main}, is complete.

\section{Projections of $1$-AD regular measures}\label{ADR}

This section contains the proof of Theorem \ref{main2}. 
\begin{definition} A Borel measure $\mu$ on $[0,1)^{d}$ is \emph{$(1,A)$-Ahlfors-David regular} -- or \emph{$(1,A)$-AD regular} in short -- if
\begin{displaymath} \frac{r}{A} \leq \mu(B(x,r)) \leq Ar \end{displaymath}
for all $x \in \spt \mu$ and $0 < r \leq \diam(\spt \mu)$. An $\calH^{1}$-measurable set $K \subset [0,1)^{d}$ is called $(1,A)$-AD regular, if $0 < \calH^{1}(K) < \infty$, and the restriction $\mu := \calH^{1}|_{K}$ of $\calH^{1}$ to $K$ is $(1,A)$-AD regular.
\end{definition}

Here is the statement of Theorem \ref{main2} once again:

\begin{thm} Given $s < 1$ and $A > 0$, there are numbers $p = p(s,A) \in \N$ and $\delta(A,s) > 0$ with the following property. Let
\begin{displaymath} S_{p} := \{e^{2\pi i k/p} : 0 \leq k < p\} \subset S^{1}. \end{displaymath}
Then, for any $(1,A)$-AD regular set $K \subset [0,1)^{d}$ with $\calH^{1}(K) \sim 1$,
\begin{equation}\label{form14} \frac{1}{p} \sum_{e \in S_{p}} N(\pi_{e}(K),\delta) \geq \delta^{-s}, \qquad 0 < \delta \leq \delta(A,s). \end{equation}
\end{thm}

The proof of Theorem \ref{main2} will use the notion of entropy, and in fact \eqref{form14} will be deduced from an intermediary conclusion of the form "the measure $\calH^{1}|_{K}$ has at least one projection with large entropy." 

\subsection{Preliminaries on entropy and projections} The presentation of this subsection follows closely that of M. Hochman's paper \cite{Ho}, although I only need a fraction of the machinery developed there. In the interest of being mostly self-contained, I will repeat some of the arguments in \cite{Ho}. 

\begin{definition}[Measures and their blow-ups in $\R^{d}$] Given a set $\Omega$, let $\calP(\Omega)$ stand for the space of Borel probability measures on $\Omega$. In what follows, $\Omega$ will be $\R^{d}$, or a cube in $\R^{d}$, and $d \in \{1,2\}$. If $Q = r[0,1)^{d} + a$ is a cube in $\R^{d}$, let $T_{Q}(x) := (x - a)/r$ be the unique homothety taking $Q$ to $[0,1)^{d}$. Given a measure $\mu \in \calP(\R^{d})$ and a cube $Q$ as above, with $\mu(Q) > 0$, define the measures
\begin{displaymath} \mu_{Q} := \frac{1}{\mu(Q)} \mu|_{Q} \in \calP(Q) \quad \text{and} \quad \mu^{Q} := T_{Q\sharp}(\mu_{Q}) \in \calP([0,1)^{d}), \end{displaymath} 
where $\mu|_{Q}$ is the restriction of $\mu$ to $Q$, and $T_{Q\sharp}$ is the push-forward under $T_{Q}$. So, $\mu^{Q}$ is a "blow-up" of $\mu_{Q}$ into $[0,1)^{d}$.
\end{definition}

\begin{definition}[Entropy]\label{entropy} Let $\mu \in \calP(\Omega)$, and let $\calF$ be a countable $\mu$-measurable partition of $\Omega$. Set
\begin{displaymath} H(\mu,\calF) := -\sum_{F \in \calF} \mu(F) \log \mu(F), \end{displaymath}
where the convention $0 \cdot \log 0 := 0$ is used. If $\calE$ and $\calF$ are two $\mu$-measurable partitions, one also defines the conditional entropy
\begin{displaymath} H(\mu,\calE | \calF) := \sum_{F \in \calF} \mu(F)  \cdot H(\mu_{F},\calE), \end{displaymath}
where, in accordance with previous notation, $\mu_{F} := \mu|_{F}/\mu(F)$, if $\mu(F) > 0$. 
 \end{definition} 
 
 The notion of conditional entropy is particularly useful, when $\calE$ \emph{refines} $\calF$, which means that every set in $\calE$ is contained in a (unique) set in $\calF$:
 \begin{proposition}[Conditional entropy formula]\label{CEF} Assume that $\calE,\calF$ are partitions as in Definition \ref{entropy}, and $\calE$ refines $\calF$. Then
 \begin{displaymath} H(\mu,\calE|\calF) = H(\mu,\calE) - H(\mu,\calF). \end{displaymath}
 In particular, $H(\mu,\calE) \geq H(\mu,\calF)$.
 \end{proposition} 
 
 \begin{proof} For $F \in \calF$, let $\calE(F) := \{E \in \calE : E \subset F\}$. A direct computation gives
 \begin{align*} H(\mu,\calE | \calF) & = - \sum_{F \in \calF} \mu(F) \cdot \sum_{E \in \calE} \mu_{F}(E) \log \mu_{F}(E)\\
 & = - \sum_{F \in \calF} \sum_{E \in \calE(F)} \mu(E) \log \frac{\mu(E)}{\mu(F)}\\
 & = - \left( \sum_{E \in \calE} \mu(E) \log \mu (E) - \sum_{F \in \calF} \log \mu(F) \sum_{E \in \calE(F)} \mu(E) \right)\\
 & = H(\mu,\calE) + \sum_{F \in \calF} \mu(F) \log \mu(F) = H(\mu,\calE) - H(\mu,\calF), \end{align*}
as claimed.  \end{proof} 

The partitions $\calE,\calF$ used below will be the dyadic partitions of $\R^{d}$: $\calE,\calF = \calD_{n}$, where $\calD_{n}$ stands for the collection of dyadic cubes of side-length $2^{-n}$. The lemma below contains two more useful and well-known -- or easily verified -- properties of entropy. The items are selected from \cite[Lemma 3.1]{Ho} and \cite[Lemma 3.2]{Ho}.
\begin{lemma}\label{factsOfLife} Let $\calE, \calF$ be countable $\mu$-measurable partitions of $\Omega$.
\begin{itemize} 
\item[(i)] The functions $\mu \mapsto H(\mu,\calE)$ and $\mu \mapsto H(\mu,\calE | \calF)$ are concave.
\item[(ii)] If $\spt \mu \subset B(0,R)$, and $f,g \colon B(0,R) \to \R$ are functions so that $|f(x) - g(x)| \leq R2^{-n}$ for $x \in B(0,R)$, then 
\begin{displaymath} |H(f_{\sharp}\mu,\calD_{n}) - H(g_{\sharp}\mu,\calD_{n})| \leq C, \end{displaymath}
where $C > 0$ only depends on $R$.
\end{itemize}
\end{lemma} 

Finally, for $n \in \N$, write $H_{n}$ for the \emph{normalised scale $2^{-n}$-entropy}
\begin{displaymath} H_{n}(\mu) := \frac{1}{\log 2^{n}} \cdot H(\mu,\calD_{n}) = \sum_{Q \in \calD_{n}} \mu(Q) \cdot \left(\frac{\log \mu(Q)}{\log 2^{-n}} \right). \end{displaymath}
This number is best interpreted as the "average local dimension of $\mu$ at scale $2^{-n}$". Now, all the definitions and tools are in place to state and prove the key auxiliary result from Hochman's paper, namely \cite[Lemma 3.5]{Ho}, in slightly modified form:
\begin{lemma}\label{multiScale} Let $\mu \in \calP([0,1)^{2})$, $e \in S^{1}$, and $m,n \in \N$ with $m < n$. Then
\begin{displaymath} H_{n}(\pi_{e\sharp}\mu) \geq \frac{m}{n} \sum_{k = 0}^{\floor*{n/m} - 1} \sum_{Q \in \calD_{km}} \mu(Q) \cdot H_{m}(\pi_{e\sharp}\mu^{Q}) - \frac{C}{m}, \end{displaymath}
where $C > 0$ is an absolute constant. 
\end{lemma}

\begin{proof} Write $n = k_{0}m + r$, where $0 \leq r < m$, and $k_{0} = \floor*{n/m}$. Then
\begin{align*} H(\pi_{e\sharp}\mu,\calD_{n}) \geq H(\pi_{e\sharp}\mu,\calD_{k_{0}m}) & = \sum_{k = 0}^{k_{0} - 1} H(\pi_{e\sharp}\mu, \calD_{(k + 1)m} | \calD_{km}) + H(\pi_{e\sharp}\mu,\calD_{0})\\
& \geq \sum_{k = 0}^{k_{0} - 1} H(\pi_{e\sharp}\mu, \calD_{(k + 1)m} | \calD_{km}) \end{align*} 
by repeated application of Proposition \ref{CEF}. Next, observe that
\begin{align*} \pi_{e\sharp}\mu & = \pi_{e\sharp} \left(\sum_{Q \in \calD_{km}} \mu|_{Q} \right) = \sum_{Q \in \calD_{km}} \pi_{e\sharp}\mu|_{Q} = \sum_{Q \in \calD_{km}} \mu(Q) \cdot \pi_{e\sharp} \mu_{Q}, \end{align*}
so, by Jensen's inequality and the concavity of (conditional) entropy,
\begin{displaymath} H(\pi_{e\sharp}\mu,\calD_{(k + 1)m} | \calD_{km}) \geq \sum_{Q \in \calD_{km}} \mu(Q) \cdot H(\pi_{e\sharp}\mu_{Q}, \calD_{(k + 1)m} | \calD_{km}). \end{displaymath}
Here
\begin{displaymath} H(\pi_{e\sharp}\mu_{Q},\calD_{(k + 1)m}|\calD_{km}) = H(\pi_{e\sharp}\mu^{Q}, \calD_{m} | \calD_{0}) = H(\pi_{e\sharp}\mu^{Q},\calD_{m}) - H(\pi_{e\sharp}\mu^{Q}, \calD_{0}), \end{displaymath}
by Proposition \ref{CEF} once again, where $H(\pi_{e\sharp}\mu^{Q},\calD_{0}) \leq 3$, because $\pi_{e\sharp}\mu^{Q}$ is supported in an interval of length $\sqrt{2}$. This leads to
\begin{align*} H_{n}(\pi_{e\sharp}\mu) & \geq \frac{1}{\log 2^{n}} \sum_{k = 0}^{k_{0} - 1} \sum_{Q \in \calD_{km}} \mu(Q) \cdot (H(\pi_{e\sharp}\mu^{Q},\calD_{m}) - 3)\\
& = \frac{m}{n} \sum_{k = 0}^{k_{0} - 1} \sum_{Q \in \calD_{km}} \mu(Q) \cdot H_{m}(\pi_{e\sharp}\mu^{Q}) - \frac{3k_{0}}{\log 2^{n}}, \end{align*}
where $3k_{0}/\log 2^{n} \leq 10/m$ as claimed. \end{proof} 

\subsection{An entropy version of Marstrand's theorem}

\begin{proposition}\label{entropyMarstrand} Assume that $\mu \in \calP([0,1)^{2})$ satisfies the linear growth condition $\mu(B(x,r)) \leq Ar$ for $x \in \R^{2}$, $r > 0$ and some $A \geq 1$. Then
\begin{displaymath} \int_{S^{1}} H_{m}(\pi_{e\sharp}\mu) \, d\sigma(e) \geq s - ACm \cdot 2^{(s - 1)m}, \qquad 0 < s < 1, \end{displaymath} 
where $\sigma$ is the unit-normalised length measure on $S^{1}$, and $C > 0$ is an absolute constant.
\end{proposition}

\begin{proof} Fix $m \in \N$. It follows from the linear growth condition for $\mu$ that
\begin{equation}\label{form12} \int_{S^{1}} 2^{m} \sum_{Q \in \calD_{m}} [\pi_{e\sharp}\mu(Q)]^{2} \, d\sigma(e) \lesssim Am, \end{equation} 
where $\calD_{m}$ now stands for the length $2^{-m}$ dyadic intervals in $\R$. This is fairly standard, so I only sketch the details: observe that for any $\nu \in \calP([0,1)^{2})$
\begin{align*} \int_{S^{1}} \|\pi_{e\sharp}\nu\|_{2}^{2} \, d\sigma(e) & = \int_{S^{1}} \int_{\R} |\hat{\nu}(te)|^{2} \, dt \, d\sigma(e)\\
& \sim \int_{\R^{2}} |\hat{\nu}(\xi)|^{2}|\xi|^{-1} \, d\xi \sim \iint \frac{d\nu x \, d\nu y}{|x - y|} =: I_{1}(\nu). \end{align*} 
Apply this with $\nu := \mu \ast \psi_{m}$, where $\psi_{m}(x) := 2^{2m}\psi(2^{m}x)$ and $\psi$ is a radial bump function with $\chi_{B(0,5)} \leq \psi \leq \chi_{B(0,10)}$. Using the linear growth condition for $\mu$, it is easy to verify that $I_{1}(\mu \ast \psi_{m}) \lesssim Am$, for $A,m \geq 1$. Further, since $\psi$ is radial, the projection $\pi_{e\sharp}(\mu \ast \psi_{m})$ has the form $(\pi_{e\sharp}\mu) \ast \phi_{m}$, where $\phi_{m}$ is a bump in $\R$ at scale $2^{-m}$, independent of $e$. Finally, the left hand side of \eqref{form12} is controlled by an absolute constant times $\|(\pi_{e\sharp}\mu) \ast \phi_{m}\|_{2}^{2}$. The inequality now follows by combining all the observations.

Let
\begin{displaymath} C_{e} := 2^{m} \sum_{Q \in \calD_{m}} [\pi_{e\sharp}\mu(Q)]^{2}. \end{displaymath}
Then, for $s < 1$ fixed,
\begin{displaymath} \pi_{e\sharp}\mu\left( \bigcup \left\{Q \in \calD_{m} : \pi_{e\sharp}\mu(Q) \geq 2^{-ms} \right\} \right) \leq C_{e}2^{(s - 1)m}, \end{displaymath}
and so
\begin{equation}\label{form13} \int_{S^{1}} \pi_{e\sharp}\mu\left( \bigcup \left\{Q \in \calD_{m} : \frac{\log \pi_{e\sharp}\mu(Q)}{\log 2^{-m}} \leq s \right\} \right) \, d\sigma(e) \lesssim Am \cdot 2^{(s - 1)m}. \end{equation}
Inspired by \eqref{form13}, let
\begin{displaymath} \calD_{m}^{e-\textup{bad}} := \left\{Q \in \calD_{m} : \frac{\log \pi_{e\sharp}\mu(Q)}{\log 2^{-m}} \leq s \right\}, \end{displaymath}
and denote by $\beta_{e}$ the total $\pi_{e\sharp}\mu$-measure of the intervals in $\calD_{m}^{e-\textup{bad}}$. Then,
\begin{align*} \int_{S^{1}} H_{m}(\pi_{e\sharp}\mu) \, d\sigma(e) & \geq \int_{S^{1}} \sum_{Q \in \calD_{m} \setminus \calD_{m}^{e-bad}} \pi_{e\sharp}\mu(Q) \left(\frac{\log \pi_{e\sharp}\mu(Q)}{\log 2^{-m}} \right)\\
& \geq \int_{S^{1}} s(1 - \beta_{e}) \, d\sigma(e) \geq s - ACm \cdot 2^{(s - 1)m}, \end{align*} 
as claimed. \end{proof}

\begin{cor}\label{entropyCor} Let $\mu$ be as in Proposition \ref{entropyMarstrand}, and let $S_{2^{m}} := \{e^{2\pi i k/2^{m}} : 0 \leq k < 2^{m}\} \subset S^{1}$. Then
\begin{displaymath} \frac{1}{|S_{2^{m}}|} \sum_{e \in S_{2^{m}}} H_{m}(\pi_{e\sharp}\mu) \geq s - AC(m \cdot 2^{(s - 1)m} + 1/m). \end{displaymath}
\end{cor} 

\begin{proof} For $e \in S_{2^{m}}$, partition $S^{1}$ into arcs $J_{e}$ of equal length $\sigma(J_{e}) := 1/|S_{2^{m}}|$ such that $e \in J_{e}$ and $|e' - e| \leq 2^{-m + 1}$ for $e' \in J_{e}$. For fixed $e \in S_{2^{m}}$, Lemma \ref{factsOfLife}(ii) then implies that
\begin{displaymath} |H_{m}(\pi_{e_{1}\sharp}\mu) - H_{m}(\pi_{e_{2}\sharp}\mu)| \lesssim \frac{1}{m}, \qquad e_{1},e_{2} \in J_{e}. \end{displaymath}
Thus, using the previous proposition,
\begin{align*} \frac{1}{|S_{2^{m}}|} \sum_{e \in S_{2^{m}}} H_{m}(\pi_{e\sharp}\mu) & = \sum_{e \in S_{2^{m}}} \int_{J_{e}} H_{m}(\pi_{e\sharp}\mu) \, d\sigma(\xi)\\
& \geq \sum_{e \in S_{2^{m}}} \int_{J_{e}} (H_{m}(\pi_{\xi\sharp}\mu) - C/m) \, d\sigma(\xi)\\
& \geq s - AC(m \cdot 2^{(s - 1)m} + 1/m), \end{align*} 
as claimed. \end{proof}

\subsection{Conclusion of the proof} Fix $0 \leq s < s' < 1$, and let $K \subset [0,1)^{2}$ be $(1,A)$-AD regular. Write $\mu := \calH|_{K}$, fix $m \in \N$, and let $S_{2^{m}} := \{e^{2\pi i k/2^{m}} : 0 \leq k < 2^{m}\} \subset S^{1}$ as in Corollary \ref{entropyCor}. A simple calculation shows that if $Q \subset [0,1)^{2}$ is a cube with $\mu(Q) > 0$, then the blow-up $\mu^{Q} \in \calP([0,1]^{2})$ satisfies the uniform linear growth condition
\begin{displaymath} \mu^{Q}(B(x,r)) \leq \left(\frac{AC\ell(Q)}{\mu(Q)}\right)r \end{displaymath}
for some absolute constant $C \geq 1$. Thus, from Lemma \ref{multiScale} and Corollary \ref{entropyCor}, one infers that, for $m < n$,
\begin{align*} \frac{1}{|S_{2^{m}}|} \sum_{e \in S_{2^{m}}} H_{n}(\pi_{e\sharp}\mu) & \geq \frac{m}{n} \sum_{k = 0}^{\floor*{n/m} - 1} \sum_{Q \in \calD_{km}} \mu(Q) \left[ \frac{1}{|S_{2^{m}}|} \sum_{e \in S_{2^{m}}} H_{m}(\pi_{e\sharp}\mu^{Q}) \right] - \frac{C}{m}\\
& \geq  \frac{m}{n} \sum_{k = 0}^{\floor*{n/m} - 1} \mathop{\sum_{Q \in \calD_{km}}}_{\mu(Q) > 0} \mu(Q) \left(s' - \frac{AC\ell(Q)}{\mu(Q)}(m \cdot 2^{(s' - 1)m} + 1/m)\right) - \frac{C}{m}\\
& = s' \cdot \frac{m}{n} \cdot \floor*{n/m} - \frac{ACm}{n} \sum_{k = 0}^{\floor*{n/m} - 1} \mathop{\sum_{Q \in \calD_{km}}}_{\mu(Q) > 0} \ell(Q)(m \cdot 2^{(s' - 1)m} + 1/m) - \frac{C}{m}. \end{align*} 
To proceed further, observe that, for any fixed generation of squares $Q$ with $\ell(Q) = r$, there are at most $AC/r$ squares $Q$ such that $\mu(Q) > 0$. Indeed, by the $(1,A)$-AD regularity of $\mu$, each square $Q$ with $\ell(Q) = r$ and $\mu(Q) > 0$ is adjacent to a square $Q'$ with $\ell(Q') = r$ and $\mu(Q') \geq r/(100A)$. Since each such "good" square $Q'$ is again adjacent to at most eight other squares $Q$ with $\mu(Q) > 0$, the claim follows. This leads to the estimate
\begin{displaymath} \frac{1}{|S_{2^{m}}|} \sum_{e \in S_{2^{m}}} H_{n}(\pi_{e\sharp}\mu) \geq s' \cdot \frac{m}{n} \cdot \floor*{n/m} - \frac{A^{2}Cm}{n} \cdot \floor*{n/m} \cdot (m \cdot 2^{(s' - 1)m} + 1/m) - \frac{C}{m}, \end{displaymath}
valid for any $0 \leq s' < 1$ and any $(1,A)$-AD regular measure $\mu \in \calP([0,1)^{2})$. Specialising to $s' := (1 + s)/2$, say, and choosing $m = m(A,s)$, where $m(A,s) \in \N$ depends only on $A$ and $s$, one obtains
\begin{displaymath} \frac{1}{|S_{2^{m}}|} \sum_{e \in S_{2^{m}}} H_{n}(\pi_{e\sharp}\mu) \geq s  \end{displaymath}
for all $n \geq n(A,s)$ (a bound depending only on $A,s$ and $m(A,s)$). Via the following lemma, this immediately leads to the desired statement about the covering numbers $N(\pi_{e}(K),\delta)$ with $\delta = 2^{-n}$, $n \geq n(A,s)$. The proof of Theorem \ref{main2} is complete.

\begin{lemma}\label{entropyAndCovering} Let $\nu \in \calP(\R^{d})$, and assume that $H_{n}(\nu) \geq s$. Then
\begin{displaymath} |\{Q \in \calD_{n} : \nu(Q) > 0\}| > 2^{nt} \end{displaymath}
for any $t < s - 1/(n \log 2)$. In particular, $N(\spt \nu,2^{-n}) \gtrsim 2^{nt}$ for such $t$.
\end{lemma}

\begin{remark} Note that the converse of the lemma is false: a large covering number certainly does not guarantee large entropy. 
\end{remark}

\begin{proof}[Proof of Lemma \ref{entropyAndCovering}] Assume that $|\{Q \in \calD_{n} : \nu(Q) > 0\}| \leq 2^{nt}$ for some $t$, and let $\calD_{n}^{\lambda-\textup{bad}}$, $\lambda \geq 0$, be the cubes $Q \in \calD_{n}$ such that $\nu(Q) \leq 2^{-\lambda n}$. Then
\begin{displaymath} \sum_{Q \in \calD_{n}^{\lambda-\textup{bad}}} \nu(Q) \leq 2^{(t - \lambda)n}, \qquad \lambda \geq t, \end{displaymath}
so that
\begin{align*} s \leq H_{n}(\nu) & = \int_{0}^{\infty} \nu\left( \bigcup \left\{Q : \frac{\log \nu(Q)}{\log 2^{-n}} \geq \lambda \right\} \right) \, d\lambda\\
& \leq t + \int_{t}^{\infty} \nu \left( \bigcup \left\{ Q : \nu(Q) \leq 2^{-\lambda n} \right\} \right) \, d\lambda\\
& \leq t + \int_{t}^{\infty} 2^{(t - \lambda)n} \, d\lambda = t + \frac{1}{n \log 2}. \end{align*}
This proves the lemma. \end{proof}

\end{document}